\documentclass[12pt, reqno]{amsart}%
\usepackage{eurosym}
\usepackage{amsmath, amsthm, amscd, amsfonts, amssymb, graphicx, color}
\usepackage[bookmarksnumbered, colorlinks, plainpages]{hyperref}%
\usepackage{amsmath}%
\usepackage{amsfonts}%
\usepackage{amssymb}%
\usepackage{graphicx}
\providecommand{\U}[1]{\protect\rule{.1in}{.1in}}
\hypersetup{colorlinks=true,linkcolor=red, anchorcolor=green, citecolor=cyan, urlcolor=red, filecolor=magenta, pdftoolbar=true}
\newtheorem{theorem}{Theorem}[section]
\newtheorem{lemma}[theorem]{Lemma}
\newtheorem{proposition}[theorem]{Proposition}

\theoremstyle{definition}

\theoremstyle{remark}

\numberwithin{equation}{section}
\begin{document}
\title[The tensor product of function algebras]{The tensor product of function algebras}
\author[Azouzi and al.]{Y.Azouzi,$^{1}$, M. A. Ben Amor$^{2}$ and J. Jaber$^{3}$$^{\ast}$}
\address{$^{1}$ {\small Research Laboratory of Algebra, Topology, Arithmetic, and
Order}\\
{\small Department of Mathematics}\\
{\small Faculty of Mathematical, Physical and Natural Sciences of Tunis}\\
{\small Tunis-El Manar University, 2092-El Manar, Tunisia}\\
$^{2}${\small Research Laboratory of Algebra, Topology, Arithmetic, and Order}\\
{\small Department of Mathematics}\\
{\small Faculty of Mathematical, Physical and Natural Sciences of Tunis}\\
{\small Tunis-El Manar University, 2092-El Manar, Tunisia}}
\email{jamel.jaber@free.fr}
\date{Received: xxxxxx; Revised: yyyyyy; Accepted: zzzzzz. }
\date{\indent$^{\ast}$ Corresponding author}
\date{Received: xxxxxx; Revised: yyyyyy; Accepted: zzzzzz. }
\date{\indent$^{\ast}$ Corresponding author}

\begin{abstract}
In this paper we study the tensor product of two $f$-algebras. We show that
the Riesz Subspace generated by a subalgebra in an $f$-algebra is an algebra 
in order to prove that the Riesz tensor product of two $f$-algebras has a
structure of an $f$-algebra. 

\end{abstract}
\maketitle

\setcounter{page}{1}


\section{Introduction.}

\bigskip The algebraic tensor product of two vector spaces $E,$ $F$ is a
vector space $E\otimes F$ with a bilinear map $\otimes:E\times
F\longrightarrow E\otimes F$ satisfying the following universal property :

for every vector space $G$ and a bilnear map $T:E\times F\longrightarrow G$
there exists a unique linear map $T^{\otimes}:E\otimes F\longrightarrow G$
such that $T=T^{\otimes}\circ\otimes.$

The Riesz tensor product of two Riesz spaces $E,$ $F$ is a Riesz space
$E\overline{\otimes}F$ with a Riesz bimorphism map $\otimes:E\times
F\longrightarrow E\overline{\otimes}F$ satisfying a similar universal property
(where maps are assumed to respect the Riesz structure). The construction of
this Tensor product is due to Fremlin in his fundamental paper \cite{F1972}.
This subject was studied several times later namely by Schaffer in \cite{SS},
Grobbler and Labaushagne \cite{GL19882, GL1988}.

It appears natural to ask about the tensor product of another class of
ordered structure: $f$-algebras! That is, if $E$ and $F$ are $f$-algebras,
does an $f$-algebra $E\overline{\otimes_{f}}F$ with a algebra and Riesz
bimorphism map $\otimes:E\times F\longrightarrow E\overline{\otimes_{f}}F$
satisfying a similar universal property (where maps are assumed to respect the
Riesz and the algebra structures) exist? 

The latter question motivated our study, therfore it is natural to focus on the
most important case of $f$-algebras: the $C(X)$'s case. The Riesz tensor
product $C(X)\overline{\otimes}C(Y)$ of $C(X)$ and $C(Y)$, where $X$ and $Y$
are topological spaces, is the Riesz sub-space of $C(X\times Y)$ generated by
the algebraic tensor product $C(X)\otimes C(Y).$ The legitimate condidate to
consider as an $f$-algebra tensor product of $C(X)$ and $C(Y)$ is the $f$-algebra
generated by $C(X)\overline{\otimes}C(Y)$ in the $f$-algebra $C(X\times Y).$
Another question araise: is $C(X)\overline{\otimes}C(Y)$ itself an
$f$-algebra? This leads to a more general question : let $B$ be a subalgebra
of the $f$-algebra $A.$ Is the Riesz subspace generated by $B$ in $A$ a subalgebra?

Section 2 gives, among other results, an affirmative question to the last
question, where we use which seems to be a new construction a Riesz space
generated by a subset.

In the next section, we present some results on bilinear maps that will be
usefull in the last section.

Section 4  is devoted to the study of the $f$-algebra tensor product. We will
prove that the Riesz tensor product of two semiprime $f$-algebras
$A$ and $B$ has automatically a structure of $f$-algebra, and satisfies the
approprate universal property that is:
\[
A\overline{\otimes}B=A\overline{\otimes_{f}}B.
\]

\section{Preliminaries}

We assume the reader to be familiar with the terminology of Riesz spaces and
$f$-algebras and we refer to \cite{Aliprantis} and \cite{SS}. Recall that a
linear mapping $T:E\longrightarrow F$ between two Riesz spaces $E$ and $F$ is
said to be lattice homomorphism if
\[
T(\left\vert x\right\vert )=\left\vert T(x)\right\vert \text{ for all }x\in
E\text{.}%
\]
A bilinear map $\varphi:E\times F\longrightarrow G$ is called positive if
$\varphi(x,y)\geq0$ for all $(x,y)\in E^{+}\times F^{+}$ and it is called a
Riesz bimorphism whenever $\varphi(.,y)$ and $\varphi(x,.)$ are Riesz
homomorphism for all $(x,y)\in E^{+}\times F^{+}$.

The (relatively) uniform topology on Riesz spaces plays a key role in the
context of this work. Let us therefore recall the definition. Let $E$ be a
Riesz space. A sequence $(f_{n})_{n\in\mathbb{N}}$ of elements of $E$ is said
to converge relatively uniformly to $f\in E$ if there exists $v\in E$ such
that for all $\varepsilon>0$ there exists $N\in\mathbb{N}$ such that
$\left\vert f_{n}-f\right\vert \leq\varepsilon.v$ for all $n\geq N$. Uniform
limits are unique if and only if $E$ is Archimedean. \textit{For this reason
all vector lattices and lattice-ordered algebras under consideration are
assumed to be Archimedean}.

The following lemma turns out to be useful for later purposes.

\begin{lemma}
\label{Cont}Let $E,F$ and $G$ be Archimedeans Riesz spaces and $\varphi
:A\times B\longrightarrow C$ be a positive bilinear map. Let $(f_{n}%
)_{n\in\mathbb{N}}\in A^{\mathbb{N}}$ and $(g_{n})_{n\in\mathbb{N}}\in
B^{\mathbb{N}}$ be two sequences which converges relatively uniformly to $f$
and $g$ respectively. Then $(\varphi(f_{n},g_{n}))_{n\in\mathbb{N}}$ converges
relatively uniformly to $\varphi(f,g)$.
\end{lemma}

\begin{proof}
There exists $u$ $\in A$ and $v\in B$ such that for all $\varepsilon>0$ there
exist $N\in\mathbb{N}$ such that $\left\vert f_{n}-f\right\vert \leq
\varepsilon.u$ and $\left\vert g_{n}-g\right\vert \leq\varepsilon.v$.
\ Observe that the sequence $(g_{n})_{n\in\mathbb{N}\text{ }}$is bounded, that
is there exist $w\in B$ such that $\left\vert g_{n}\right\vert \leq w$ for all
$n\in\mathbb{N}$. It follows from the bilinearity and the positivity of
$\varphi$ that for all $n\geq N$,
\begin{align*}
\left\vert \varphi(f_{n},g_{n})-\varphi(f,g)\right\vert  & \leq\varphi
(\left\vert f_{n}-f\right\vert ,\left\vert g_{n}\right\vert )+\varphi
(\left\vert f\right\vert ,\left\vert g_{n}-g\right\vert )\\
& \leq\varepsilon.(\varphi(u,w)+\varphi(\left\vert f\right\vert ,v))\text{.}%
\end{align*}
This implies the desired result.
\end{proof}

\bigskip The next paragraph deals with the notion of function-algebras (or $f
$-algebras). A Riesz space $A$ is called a lattice ordered algebra if there
exists an associative multiplication in $A$ with the usual algebraic
properties such that $fg\geq0$ for all $f,g\in A^{+}$. The lattice-ordered
algebra $A$ is said to be an $f$-algebra if $f\wedge g=0$ and $0\leq h\in A$
imply $fh\wedge g=hf\wedge g=0$. The most classical example of an $f$-algebra
is the algebra $C(X)$ of all real-valued continuous functions on a topological
space $X$. \ Recall from \cite[Theorem 142.7]{Z1983} that if $A$ is an
$f$-algebra with unit element $e$ then for all $0\leq x\in A$, the sequence
$(x\wedge ne)_{n\in\mathbb{N}}$ converge uniformly to $x$. The reader can
consult \cite{Z1983} for more information about $f$-algebras.

Finally we recall some facts about the universally completion of a Riesz
space. A Dedekind complete vector lattice is called universally complete
whenever every set of pairwise disjoint positive elements has a supremum.
Every vector lattice $E$ has a universal completion $E^{u}$, meaning that
there exists a unique universally complete vector lattice $E^{u}$ such that
$E$ can be identified with an order dense vector sublattice of $E^{u}$.
Moreover, if $e$ is a weak order unit in $E$ then $E^{u}$ can be endowed with
a multiplication, in such a manner that $E^{u}$ becomes an $f$-algebra with
$e$ as identity.

\section{Riesz subspace generated by a subalgebra.}

\bigskip Let $D$ be a subset of a Riesz space $E$. The Riesz subspace
generated by $D$, that is the smallest Riesz subspace containing $D$, will be
denoted by $\mathcal{R}(D)$. The first theorem of this section furnishes a
useful new construction, as the best of our knowledge, of the Riesz subspace
generated by a vector subspace. Consider a vector subspace $F$ of $E $. Define
a sequence $(L_{n})_{n\in\mathbb{N}}$ of vector subspaces by:

$L_{1}=F$ and $L_{n+1}$ is the vector subspace generated by $L_{n}$ and
$|L_{n}|$ for $n=1,2,\ \ldots$ where
\[
|D|=\{|x|:x\in D\}
\]
for any subset $D$ of $E$.

\begin{theorem}
\label{C}Let $F$ be a vector subspace of a Riesz space $E$, and $\left(
L_{n}\right)  _{n\in\mathbb{N}}$ defined as above. Then $\mathcal{R}(F)=%
{\textstyle\bigcup\limits_{n=1}^{\infty}}
L_{n}$.
\end{theorem}

\begin{proof}
Since $(L_{n})_{n\in\mathbb{N}}$ is an increasing sequence of vector
subspaces, it follows that $L=\bigcup_{n=1}^{\infty}L_{n}$ is a vector
subspace of $E$. Moreover, for all $f\in L_{n}$ we have $\left\vert
f\right\vert \in L_{n+1}\subseteq L.$ This implies that $L$ is a Riesz
subspace which contains $F$. Hence $\mathcal{R}(F)\subseteq L$. Conversely, we
prove easily by induction that $L_{n}\subseteq\mathcal{R}(F)$ for all
$n=1,2,..$which yields to the inclusion $L\subseteq\mathcal{R}(F).$
\end{proof}

\bigskip The following technical lemma will be used to prove the main theorem
of this section.

\begin{lemma}
\label{L1}\textit{Let} $A$ \textit{be a semiprime} $f$-\textit{algebra. Then
the following hold}

\begin{description}
\item[(i)] \textit{For all} $a\in A^{+}$
\[
0\leq a^{2}\leq a+a^{3}.
\]

\item[(ii)] \textit{For all} $a,\ b\in A$,
\[
a^{+}b^{+}=(ab)^{+}\wedge((a+a^{3})^{+}+(b+b^{3})^{+})\ .
\]

\end{description}
\end{lemma}

\begin{proof}
(i) Assume first that $A$ has a unit element $e$. Then
\begin{align*}
a^{3}+a-a^{2} &  =a(a^{2}-a+e)\\
&  =a((a-2^{-1}e)^{2}+(3/4)e)\geq0
\end{align*}
as squares in $f$-algebra are positive. Since any semiprime $f$-algebra can be
embedded as a Riesz subalgebra in an $f$-algebra with unit, the result follows
for general case.

(ii) Let $a,\ b\in A$ and put $z=(a+a^{3})^{+}+(b+b^{3})^{+}.$ Observe that
\[
z=a^{+}+\left(  a^{+}\right)  ^{3}+b^{+}+\left(  b^{+}\right)  ^{3}.
\]
It follows from $\left(  i\right)  $ that
\begin{equation}
z\geq(a^{+})^{2}+(b^{+})^{2}\geq a^{+}b^{+}%
\end{equation}

Moreover, since $a^{-}b^{-}\wedge z=0$ we get
\begin{equation}
(ab\ )^{+}\wedge z=(a^{+}b^{+}+a^{-}b^{-})\wedge z=a^{+}b^{+}\wedge
z=a^{+}b^{+}%
\end{equation}

\end{proof}

The subalgebra generated by a Riesz subspace of an $f$-algebra $A$ fails in
general to be a Riesz subspace. Take for example $A=C[0,1]$, and $B$ the
vector subspace generated by $u$ where $u(x)=x$ for all $x\in\lbrack0,1].$
Surprising enough, it turns out that the converse situation is true. In fact,
the main result of this section states that the Riesz subspace generated by a
subalgebra of an $f$-algebra $A$ is indeed a subalgebra.

\begin{theorem}
\label{M1}\textit{Let} $A$ \textit{be a semiprime} $f$\textit{-gebra and} $B$
\textit{be a subalgebra of} $A$. \textit{Then the Riesz subspace}
$\mathcal{R}(B)$ \textit{generated by} $B$ \textit{is an} $f$%
-\textit{subalgebra of} $A$.
\end{theorem}

\begin{proof}
We use the same construction as in Theorem 2.1 considering $B$ instead of
$F$. Then $\mathcal{R}(B)=\bigcup_{n=1}^{\infty}L_{n}$. All that remains is to prove
that $\mathcal{R}(B)$ is closed under multiplication. It suffices to show that
for all $n\in\mathbb{N}$ we have the following property%
\[
\mathcal{P}_{n}:\text{for all }f,\ g\in L_{n},\ fg\text{ and }f|g|\text{ are
in }\mathcal{R}(B).
\]

We will proceed by induction on $n$.

Let us show that $\mathcal{P}_{1}$ is true. Take $f,\ g\in L_{1}=B$. Then $fg
$ belongs to $B$ because $B$ is a subalgebra of $A$. We claim that
$f|g|\in\mathcal{R}(B)$. Observe first that $(f+f^{3})^{+},\ (g+g^{3})^{+}$
and $(fg)^{+}$ are in $\mathcal{R}(B)$. According to (ii) in Lemma \ref{L1} ,
we get $f^{+}g^{+}\in\mathcal{R}(B)$ . Since $f^{+}g^{-}=f^{+}(-g)^{+}%
\in\mathcal{R}(B)$ and $f^{-}g^{-}=(-f)^{+}(-g)^{+}\in\mathcal{R}(B)$ we
derive that $f|g|\in\mathcal{R}(B)$.

Suppose now that $\mathcal{P}_{n}$ is true and pick two elements $f$ and $g$
in $L_{n+1}$. Then $f$ and $g$ can be written as a finite sums
\begin{equation}
f=\sum_{\alpha}f_{\alpha}+\sum_{\beta}|f_{\beta}|\ \text{and}\ g=\sum_{\gamma
}g_{\gamma}+\sum_{\delta}|g_{\delta}|\text{ }\label{eq3}%
\end{equation}
where $f_{\alpha},\ f_{\beta},\ g_{\gamma}$ and $g_{\delta}$ belong to $L_{n}%
$. Applying the inductive hypothesis to these elements, a direct calculation
shows therefore that $f.g\in\mathcal{R}(B)$ . It remains to prove that
$f|g|\in\mathcal{R}(B)$. This is obvious in the case when $f$ is positive
because $f|g|=|fg|\in\mathcal{R}(B)$. In the general case, using the
decomposition of $f$ in (\ref{eq3}), it is sufficient to show that $f_{\alpha
}|g|\in\mathcal{R}(B)$ and $|f_{\beta}||g|\in\mathcal{R}(B)$. But $|f_{\beta
}||g|=|f_{\beta}g|\in\mathcal{R}(B)$ and since $f_{\alpha}^{+}$ and
$f_{\alpha}^{-}$ are in $L_{n+1}$ for all $\alpha$ we get by the positive
case
\[
f_{\alpha}.\ |g|=f_{\alpha}^{+}|g|-f_{\alpha}^{-}|g|\in\mathcal{R}(B)
\]
The proof is now complete.
\end{proof}

The next result is an operator-version of Theorem \ref{M1}. In fact if we
consider a Riesz homomorphism $T$ from $\mathcal{R}(B)$ to an $f$-algebra $C$
then $T$ is an algebra homomorphism if and only if its restriction to $B$ is
an algebra homomorphism.

\begin{theorem}
\label{MUP}\textit{Let} $B$ \textit{be a subalgebra of a semiprime}
$f$-\textit{algebra} $A$ \textit{and} $C$ \textit{be an} $f$-\textit{algebra.
Let} $T$ : $\mathcal{R}(B)\rightarrow C$ \textit{be a Riesz homomorphism.
Then}, $T$ \textit{is multiplicative if and only if} $T$ \textit{is
multiplicative on} $B$, \textit{that is} $T(fg)=T(f)T(g)$ \textit{for all}
$f,\ g\in B$.
\end{theorem}

\begin{proof}
Only one implication requires proof. Assume then that $T(fg)=T(f)T(g)$ for all
$f,\ g\in B$. We proceed as in proof of Theorem \ref{M1} by proving a similar
property to $\mathcal{P}_{n}$ which is the following%
\[
\mathcal{P}_{n}^{\prime}:\text{for all }f,\ g\in L_{n},\ T(fg)=TfTg\text{ and
}T(f\ |g|)=T(f)|Tg|\text{.}%
\]

Pick tow elements $f\ $and $g$ in $B$. Since $T$ is multiplicative on $B$, we
get $T(fg)=T(f)T(g)$. We claim that $T(f\ |g|)=T(f)T(|g|)$. Using the fact
that $T$ is a Riesz homomorphism and according to Lemma \ref{L1} we can write%
\begin{align*}
T(f^{+}g^{+}) &  =T((fg)^{+})\wedge((T(f)+T(f^{3}))^{+}+(T(g)+T(g^{3}))^{+})\\
&  =(T(f)T(g))^{+}\wedge((T(f)+T(f)^{3})^{+}+(T(g)+T(g)^{3})^{+})\\
&  =T(f)^{+}T(g)^{+}.
\end{align*}

Slight modifications in the above equalities yields to:
\[
T(f^{-}g^{+})=T(f)^{-}T(g)^{+},\ T(f^{-}g^{+})=T(f)^{-}T(g)^{+}\text{ and
}T(f^{+}g^{-})=T(f)^{+}T(g)^{-}%
\]

This implies that
\[
T(f|g|)=T(f)T(|g|).
\]

So $\mathcal{P}_{1}^{\prime}$ is true. It is clear now that the rest of the
proof is similar to the previous one.
\end{proof}

\section{Some results on bilinear maps.}

In this section we will collect some proprieties of bilinear maps between
$f$-algebras. Recall that for a map $S:X\longrightarrow Y$, where $X$ and $Y$
are two vector spaces, the zero-set of $S$ is the set
\[
N(S)=\left\{  x\in X:S(x)=0\right\}  \text{.}%
\]

It is well known that if $\varphi$ and $\psi$ are two linear functionals on a
vector space $E$ such that $N\left(  \varphi\right)  \subseteq N\left(
\psi\right)  ,$ then there exists a scalar $\lambda$ such that $\psi
=\lambda\varphi$. The following lemma shows that this result can be extended
to bilinear functionals.

\begin{lemma}
\label{AZ}Let $E$ and $F$ be two vector spaces, $\varphi,\psi:E\times
F\longrightarrow\mathbb{R}$ be two bilinear functionals on $E\times F$ such
that $N\left(  \varphi\right)  \subseteq N\left(  \psi\right)  $ then
$\psi=\lambda.\varphi$ for some $\lambda\in\mathbb{R}$.
\end{lemma}

\begin{proof}
If $\varphi=0$ the result is obvious. Otherwise, pick a fixed element $e\in E
$ and consider the linear functionals $\varphi_{e}=\varphi\left(  e,.\right)
$ and $\psi_{e}=\psi\left(  e,.\right)  .$ Since $N\left(  \varphi_{e}\right)
\subseteq N\left(  \psi_{e}\right)  $, there exists some real number
$\lambda\left(  e\right)  $ such that
\begin{equation}
\psi\left(  e,h\right)  =\lambda\left(  e\right)  .\varphi\left(  e,h\right)
\text{ for all }h\in E\text{.}\label{A}%
\end{equation}

We claim that $\lambda(e)$ does not depend on $e$. If $\varphi_{e}=0,$ then
$\psi_{e}=0$ and any real $\lambda\left(  e\right)  $ is convenient. It is
sufficient then to show that $\lambda\left(  e\right)  =\lambda\left(
f\right)  $ for all $e,f$ such that $\varphi_{e}$ and $\varphi_{f}$ are not
null linear functionals. In this case there exists $g\in F$ such that
$\varphi\left(  e,g\right)  \neq0$ and $\varphi\left(  f,g\right)  \neq0$ ($E$
cannot be the union of two strict subspaces). Put $\alpha=\dfrac
{\varphi\left(  e,g\right)  }{\varphi\left(  f,g\right)  }$ and observe that
$\varphi\left(  e-\alpha f,g\right)  =0.$ So $\psi\left(  e-\alpha f,g\right)
=0$ which implies that%
\[
\psi\left(  e,g\right)  =\alpha\psi\left(  f,g\right)  .
\]
We obtain
\[
\lambda\left(  e\right)  \varphi\left(  e,g\right)  =\alpha\lambda\left(
f\right)  \varphi\left(  f,g\right)  =\lambda\left(  f\right)  \varphi\left(
e,g\right)
\]
which proves that $\lambda(e)=\lambda(f)$ and we are done.
\end{proof}

As a consequence of the previous lemma we obtain the following result.

\begin{proposition}
\label{AZZ}Let $A_{1}$ and $A_{2}$ be $f$-algebras with unit elements $e_{1}$
and $e_{2},$ respectively, and let $T:A_{1}\times A_{2}\longrightarrow
\mathbb{R}$ be a Riesz bimorphism such that $T\left(  e_{1},e_{2}\right)  =1$.
Then $T$ is multiplicative.
\end{proposition}

\begin{proof}
Let $a_{i}\in A_{i}$ such that $0\leq a_{i}\leq e_{i},$ $i=1,2$. We define the
bilinear functional $S$ by
\[
S\left(  x_{1},x_{2}\right)  =T\left(  a_{1}x_{1},a_{2}x_{2}\right)  \text{
for all }\left(  x_{1},x_{2}\right)  \in A_{1}\times A_{2}\text{.}%
\]

It follows from the following%
\[
\left\vert T\left(  a_{1}x_{1},a_{2}x_{2}\right)  \right\vert =T\left(
\left\vert a_{1}x_{1}\right\vert ,\left\vert a_{2}x_{2}\right\vert \right)
\leq T\left(  \left\vert x_{1}\right\vert ,\left\vert x_{2}\right\vert
\right)  =\left\vert T\left(  x_{1},x_{2}\right)  \right\vert
\]
that $N(T)\subseteq N(S)$. According to Lemma \ref{AZ} there exists a real
number $\lambda$ such that $S=\lambda T$. By hypothesis,%
\[
\lambda=\lambda.T(e_{1},e_{2})=S(e_{1},e_{2})=T(a_{1},a_{2}).
\]
So one has
\begin{equation}
T(a_{1}x_{1},a_{2}x_{2})=T(a_{1},a_{2}).T(x_{1},x_{2})\text{ for all }x_{1}\in
A_{1}\text{, }x_{2}\in A_{2}\label{3}%
\end{equation}

Now if $a_{1}$ and $a_{2}$ are positive, we get for all $n=1,2,..$ and
$x_{1}\in A_{1},$ $x_{2}\in A_{2},$
\[
T((a_{1}\wedge ne_{1})x_{1},(a_{2}\wedge ne_{2})x_{2})=T\left(  a_{1}\wedge
ne_{1},a_{2}\wedge ne_{2}\right)  T(x_{1},x_{2})
\]
The equality
\[
T(a_{1}x_{1},a_{2}x_{2})=T(a_{1},a_{2}).T(x_{1},x_{2})\text{ for all }x_{1}\in
A_{1}\text{ and }x_{2}\in A_{2}\text{.}%
\]
can be deduced by relatively uniform continuity technics and Lemma \ref{Cont}.
\end{proof}

The next theorem gives the connection between Riesz bimorphisms and
multiplicative bilinear maps on $f$-algebras. This theorem is proved in
\cite[Theorem 1]{Boulabiar} and we give here an alternative proof which can be
adapted to multilinear maps.

\begin{theorem}
\label{Kar}Let $A_{1},$ $A_{2}$ and $B$ be $f$-algebras with unit elements
$e_{1},e_{2}$ and $e_{B}$, respectively and $T:A_{1}\times A_{2}%
\longrightarrow B$ be a lattice bimorphism satisfying $T\left(  e_{1}%
,e_{2}\right)  =e_{B}$. Then $T$ is multiplicative bilinear map.
\end{theorem}

\begin{proof}
Let $a_{i},x_{i}\in A_{i}$ for $i=1,2$ and consider the the $f$-subalgebra
$C_{i}$ of $A_{i}$ generated by $\left\{  e_{i},x_{i},a_{i}\right\}  .$ Let
$D$ be the $f$-subalgebra of $B$ generated by $T\left(  C_{1}\times
C_{2}\right)  $ which is finitely generated. Take now an arbitrary
multiplicative positive functional $\omega$ on $D$. Applying Proposition
\ref{AZZ} to the bilinear functional $\omega\circ T:C_{1}\times C_{2}%
\longrightarrow\mathbb{R}$, we get
\[
\omega\left(  T\left(  a_{1}x_{1},a_{2}x_{2}\right)  \right)  =\omega\left(
T\left(  a_{1},a_{2}\right)  \right)  .\omega\left(  T\left(  x_{1}%
,x_{2}\right)  \right)  =\omega\left(  T\left(  a_{1},a_{2}\right)  T\left(
x_{1},x_{2}\right)  \right)
\]

Since the set of all real-valued multiplicative lattice homomorphisms on $D$
separates the points of $D$ (see \cite[Corollary 7]{BP}) we obtain%
\[
T\left(  a_{1}x_{1},a_{2}x_{2}\right)  =T\left(  a_{1},a_{2}\right)  T\left(
x_{1},x_{2}\right)  \text{ for all }a_{1},x_{1}\in A_{1}\text{ and }%
a_{2},x_{2}\in A_{2}%
\]

This concludes the proof.
\end{proof}

\section{The Riesz tensor product of $f$-algebras}

The tensor product of two $f$-algebras has not been studied before. It is
natural to think to endow the Riesz tensor product $A\overline{\otimes}B$ of
$f$-algebras with an $f$-algebra structure. A first way is to extend the
natural multiplication on the algebraic tensor product $A\otimes B$ to
$A\overline{\otimes}B.$ Another approach consists to embed $A\overline
{\otimes}B$ in some $f$-algebra $C$ and study the stability of $A\overline
{\otimes}B $ by multiplication. Focus first on the special and important
$C\left(  X\right)  $'s case.

Let $X$ and $Y$ be compact spaces. The algebraic tensor product $C(X)\otimes
C(Y)$ of the function-algebras $C(X)$ and $C(Y)$ can be viewed as the vector
subspace of $C(X\times Y)$ consisting of all functions $h(s,t)=\sum
f_{i}(s).g_{i}(t)$ where $f_{i}\in C(X)$ and $g_{i}\in C(Y)$ . The Riesz
tensor product $C(X)\overline{\otimes}C(Y)$ is the Riesz subspace of
$C(X\times Y)$ generated by the algebra $C(X)\otimes C(Y)$ (see for instance
\cite{S1980}). The following result is now an immediate consequence of Theorem
\ref*{M1}.

\begin{theorem}
Let $X$ and $Y$ be compact spaces. Then the Riesz tensor product
$C(X)\overline{\otimes}C(Y)$ of $C(X)$ and $C(Y)$ is an $f$-subalgebra of
$C(X\times Y)$.
\end{theorem}

In the sequel we study the general case. Let $A$ and $B$ be two semiprime
$f$-algebras. The algebraic tensor product $A\otimes B$ can be furnished in a
canonical way with an algebra product satisfying%
\[
(a\otimes b).(a^{\prime}\otimes b^{\prime})=aa^{\prime}\otimes bb^{\prime
}\text{ for all }a,a^{\prime}\in A\text{ and }b,b^{\prime}\in B.
\]

We start our discussion about the proprities of the Riesz tensor product with
this lemma.

\begin{lemma}
\label{L2}Let $A$ and $B$ be two $f$-algebras with unit elements $e_{A}$ and
$e_{B}$ respectively, then $e_{A}\otimes e_{B}$ is a weak order unit in
$A\overline{\otimes}B$.
\end{lemma}

\begin{proof}
Let $0\leq u$ in $A\overline{\otimes}B$ such that $u\wedge e_{A}\otimes
e_{B}=0$. We have to show that $u=0$. We argue by contradiction and we suppose
that $u\neq0$. By \cite[Theorem 5.8. (e)]{GL1988}, there exist $x\in A_{+}$
and $y\in B_{+}$ such that%
\[
0<x\otimes y\leq u\text{.}%
\]
Therefore
\[
0\leq x\wedge e_{A}\otimes y\wedge e_{B}\leq u\wedge e_{A}\otimes e_{B}=0
\]

So $x\wedge e_{A}\otimes y\wedge e_{B}=0$ and then $x\wedge e_{A}=0$ or
$y\wedge e_{B}=0.$ Hence $x=0$ or $y=0$ and we obtain%
\[
x\otimes y=0
\]
which is a contradiction.
\end{proof}

It follows from Lemma \ref{L2} that the universal completion $(A\overline
{\otimes}B)^{u}$ of $A\overline{\otimes}B$ can be equipped with an $f$-algebra
multiplication denoted by $\star$ such that $e_{A}\otimes e_{B}$ is the unit
element. The next result shows that $\star$ extends the canonical
multiplication on $A\otimes B.$

\begin{proposition}
\label{alg}Let $A$ and $B$ be tow $f$-algebras with unit $e_{A}$ and $e_{B}$
respectively. Then $u\star v=u.v$ for all $u,v\in A\otimes B.$ So $A\otimes B
$ can be considered as a subalgebra of $\left(  A\overline{\otimes}B\right)
^{u}.$
\end{proposition}

\begin{proof}
Let $\sigma$ be the canonical mapping defined by:%
\[%
\begin{array}
[c]{ccccc}%
\sigma & : & A\times B & \longrightarrow & (A\bar{\otimes}B)^{u}\\
&  & (x,y) & \longmapsto & x\otimes y
\end{array}
\]
Since $\sigma$ is a Riesz bimorphism and $\sigma(e_{A},e_{B})=e_{A}\otimes
e_{B}$ it follows from Theorem \ref{Kar} that $\sigma$ is multiplicative.
Consequently,
\[
\sigma((x,y).(x^{\prime},y^{\prime}))=\sigma(x,y)\star\sigma(x^{\prime
},y^{\prime})\text{.}%
\]
for all $x,x^{\prime}\in A$ and $y,y^{\prime}\in B$, which means that%
\[
xx^{\prime}\otimes yy^{\prime}=(x\otimes y)\star(x^{\prime}\otimes y^{\prime
})\text{.}%
\]
This yields to%
\[
(x\otimes y).(x^{\prime}\otimes y^{\prime})=(x\otimes y)\star(x^{\prime
}\otimes y^{\prime})\text{.}%
\]
The result follows immediately by bilinearity.
\end{proof}

The previous lemma enables us to make a connection between our tow approaches
mentioned in the beginning of this section and leads us to prove the main
theorem of this section.

\begin{theorem}
\label{Main2}Let $A$, $B$ be two $f$-algebras with unit elements $e_{A}$ and
$e_{B}$ respectively. Then the canonical multiplication on the algebraic
tensor product $A\otimes B$ can be extended to a multiplication on the Riesz
tensor product $A\overline{\otimes}B$, such that $A\overline{\otimes}B$ is an
$f$-algebra with unit. Moreover, $A\overline{\otimes}B$ satisfy the
multiplicative universal property (MUP):\newline For all multiplicative and
Riesz bimorphism $\psi:A\times B\longrightarrow C$ where $C$ is an
$f$-algebra, there exists a unique algebra and Riesz homomorphism
$\psi^{\otimes}:A\overline{\otimes}B\longrightarrow C$ such that
\[
\psi^{\otimes}(f\otimes g)=\psi(f,g)\text{ for all }f\in A\text{, }g\in
B\text{.}%
\]

\end{theorem}

\begin{proof}
Observe first that $A\overline{\otimes}B$ is the Riesz subspace $R(A\otimes
B)$ generated by $A\otimes B$ in $(A\overline{\otimes}B)^{u}$. Keeping in mind
that $A\otimes B$ is a subalgebra $(A\overline{\otimes}B)^{u}$ (Proposition
\ref{alg}), it follows from Theorem \ref{M1} that $A\overline{\otimes}B$ is an
$f$-algebra with unit element $e_{A}\otimes e_{B}.$ It remains to prove that
$A\overline{\otimes}B$ satisfies the MUP. Let $C$ be an $f$-algebra and
$\psi:A\times B\longrightarrow C$ be a multiplicative Riesz bimorphism, then
there exists a lattice homomorphism $\psi^{\otimes}:A\overline{\otimes
}B\longrightarrow C$ such that $\psi^{\otimes}(x\otimes y)=\psi(x,y)$ for all
$x\in A$, $y\in B$. Since $\psi$ is multiplicative then the restriction of
$\psi^{\otimes}$ to $A\otimes B$ is multiplicative, it follows from Theorem
\ref{MUP} that $\psi^{\otimes}$ is multiplicative on the whole space.
\end{proof}

Our last result concerns the Riesz tensor product of semiprime $f$-algebras.
It is well known that any semiprime $f$-algebra $C$ can be seen as
$f$-subalgebra of $\mathrm{Orth}(C),$ where $\mathrm{Orth}(C)$ indicates the
unital $f$-algebra of all orthomorphisms on $C$.

\begin{theorem}
Let $A$ and $B$ be two semiprime $f$-algebras. Then the Riesz tensor product
$A\overline{\otimes}B$ can be endowed with an $f$-algebra product, which
extends the algebraic multiplication. In addition $A\overline{\otimes}B$
satisfies the following multiplicative universal property (MUP): For all
multiplicative and Riesz bimorphism $\psi:A\times B\longrightarrow C,$ where
$C$ is an $f$-algebra, there exists a unique algebra and Riesz homomorphism
$\psi^{\otimes}:A\overline{\otimes}B\longrightarrow C$ such that
\[
\psi^{\otimes}(f\otimes g)=\psi(f,g)\text{ for all }f\in A\text{, }g\in
B\text{.}%
\]

\end{theorem}

\begin{proof}
It follows from \cite[Corollary 4.5]{F1972} that $A\overline{\otimes}B$ is the
Riesz subspace generated by $A\otimes B$ in $\mathrm{Orth}(A)\overline
{\otimes}\mathrm{Orth}(B)$. By Theorem \ref{Main2} $\mathrm{Orth}%
(A)\overline{\otimes}\mathrm{Orth}(B)$ is an $f$-algebra. According to Theorem
\ref{M1}, $A\overline{\otimes}B$ is an $f$-subalgebra of $\mathrm{Orth}%
(A)\overline{\otimes}\mathrm{Orth}(B)$. The proof of MUP is similar to the one
of Theorem \ref{Main2}.
\end{proof}

\end{document}